\documentclass[12pt,oneside]{amsart}

\usepackage{amsmath}
\usepackage{amssymb}
\usepackage{amscd}

\title{A hyperplane section theorem for Galois points and its application}
\author{Satoru Fukasawa}

\subjclass[2000]{14J70, 12F10}
\keywords{Galois point, hypersurface}
\address{Department of Mathematical Sciences, Faculty of Science, Yamagata University,  
Kojirakawa-machi 1-4-12, Yamagata 990-8560, Japan}
\email{s.fukasawa@sci.kj.yamagata-u.ac.jp} 
\thanks{The author was partially supported by JSPS KAKENHI Grant Numbers 22740001 and 16K05088.} 

\newtheorem{theorem}{Theorem}
\newtheorem{proposition}{Proposition}
\newtheorem{corollary}{Corollary}
\newtheorem{lemma}{Lemma}

\theoremstyle{definition}

\newtheorem{remark}{Remark}

\begin{document}
\begin{abstract} 
A point $P$ in projective space is said to be Galois with respect to a hypersurface if the function field extension induced by the projection from $P$ is Galois. 
We present a hyperplane section theorem for Galois points. 
 Precisely, if $P$ is a Galois point for a hypesurface, then $P$ is Galois for a general hyperplane section passing through $P$. 
As an application, we determine hypersurfaces of dimension $n$ with $n$-dimensional sets of Galois points. 
\end{abstract}
\maketitle

\section{Introduction}  
Let the base field $K$ be an algebraically closed field of characteristic $p \ge 0$ and let $X \subset \mathbb P^{n+1}$ be an irreducible and reduced hypersurface of dimension $n$ and of degree $d$. 
H. Yoshihara introduced the notion of the {\it Galois point} (see \cite{fukasawa1, miura-yoshihara, yoshihara1, yoshihara2, yoshihara3}). 
If the function field extension $K(X)/K(\mathbb P^n)$ induced by the projection $\pi_P:X \dashrightarrow \mathbb P^n$ from a point $P \in \mathbb P^{n+1}$ is Galois, then the point $P$ is said to be Galois.

Galois point theory has given a new viewpoint of classification of algebraic varieties, by the distribution of Galois points (see \cite{fukasawa1, fukasawa2, fukasawa-hasegawa, fukasawa-takahashi, miura-yoshihara, takahashi1, yoshihara1, yoshihara2, yoshihara3}). 
If the dimension of the singular locus $S_X$ is at most $n-2$, then Fukasawa and Takahashi \cite{fukasawa-takahashi} presented upper bounds for the number of Galois points, as a generalization of a result of Yoshihara for smooth hypersurfaces \cite{yoshihara3}.  
To do this, they showed a ``hyperplane section theorem'' (see \cite[Theorem 1.3]{fukasawa-takahashi}). 
In this article, we prove this theorem for {\it arbitrary} Galois points (which may be singular) with respect to {\it arbitrary} hypersurfaces. 
The intersection of (almost) all  tangent spaces of $X$ is denoted by $T_X$. 

\begin{theorem} \label{HyperplaneMain} 
Let $X \subset \mathbb P^{n+1}$ be a hypersurface of dimension $n \ge 2$ and degree $d \ge 3$ in characteristic $p \ge 0$, and let $P$ be a Galois point for $X$ with multiplicity $m$, where $0 \le m \le d-2$.  
Then: 
\begin{itemize}
\item[(i)] A general hyperplane $H$ passing through $P$ satisfies the following condition: 
\begin{itemize}
\item[$(\star)$] the hyperplane section $X_H:=X \cap H$ is an irreducible hypersurface in $H \cong \mathbb P^n$ of degree $d$ such that $S_{X_H}=S_X \cap H$, $P \not\in T_{X_H}$, and the multiplicity of $X_H$ at $P$ is equal to $m$. 
\end{itemize}
\item[(ii)] Let $H$ be a hyperplane passing through $P$ and satisfying the condition $(\star)$.
Then, the point $P$ is Galois for $X_H$. 
\item[(iii)] In this case, the Galois groups are isomorphic: $G_P(X) \cong G_P(X_H)$.  
\end{itemize}
\end{theorem} 

As an application, we generalize results of Fukasawa and Hasegawa \cite{fukasawa2, fukasawa-hasegawa} for plane curves with infinitely many Galois points. 
Let $\Delta(X)$ (resp. $\Delta'(X)$) be the set of all Galois points contained in $X \setminus S_X$ (resp. $\mathbb P^{n+1} \setminus X$). 

\begin{theorem} \label{InnerMain}
Let $X \subset \mathbb P^{n+1}$ be a hypersurface of dimension $n \ge 1$ and degree $d \ge 4$ in characteristic $p \ge 0$. 
Then, the following conditions are equivalent.  
\begin{itemize} 
\item[(i)] There exists a non-empty Zariski open subset $U$ of $X$ such that $U \subset \Delta(X)$. 
\item[(ii)] $p>0$, $d=p^e$ for some $e>0$, and $X$ is projectively equivalent to the hypersurface defined by $X_0^{p^e-1}X_1-X_2^{p^e}=0$. 
\end{itemize} 
In this case, $\Delta(X)=X \setminus \{X_0=X_2=0\}$, and the induced Galois group $G_P$ is a cyclic group of order $p^e-1$ for any point $P \in \Delta(X)$. 
\end{theorem}

\begin{theorem} \label{OuterMain}
Let $X \subset \mathbb P^{n+1}$ be a hypersurface of dimension $n \ge 1$ and degree $d \ge 3$ in characteristic $p \ge 0$. 
Then, the following conditions are equivalent. 
\begin{itemize} 
\item[(i)] There exist an irreducible Zariski closed subset $Y \subset \mathbb P^{n+1}$ of dimension $n$ and a non-empty open subset $U_Y$ of $Y$ such that $U_Y \subset \Delta'(X)$. 
\item[(ii)] $p>0$, $d=p^e$ for some $e>0$, and $X$ is projectively equivalent to an irreducible hypersurface whose equation is of the form $$\sum_{j=0}^{e}\sum_{i=1}^{n+1}\alpha_{ij}X_0^{p^e-p^j}X_i^{p^{j}}=0, $$
where $\alpha_{ij} \in K$.  
\end{itemize} 
In this case, $\Delta'(X)$ is a Zariski open set of a hyperplane (see Proposition \ref{distribution}), and the induced Galois group $G_P$ is isomorphic to $(\mathbb Z/p\mathbb Z)^{\oplus e}$ for any point $P \in \Delta'(X)$. 
\end{theorem}

\section{Preliminaries and Lemmas} 
Let $X \subset \mathbb P^{n+1}$ be an irreducible hypersurface, let $S_X$ be the singular locus of $X$, and let $X_{\rm sm}=X \setminus S_X$ be the smooth locus. 
The projective tangent space at a smooth point $P \in X\setminus S_X$ is denoted by $T_PX \subset \mathbb P^{n+1}$. 
Let $\check{\mathbb P}^{n+1}$ be the dual projective space which parameterizes hyperplanes of $\mathbb P^{n+1}$. 
The Gauss map $\gamma$ of $X$ is a rational map from $X$ to $\check{\mathbb P}^{n+1}$ which sends a smooth point $P$ to the tangent space $T_PX$. 
If $F$ is the defining polynomial of $X$, then $\gamma$ is given by $(\partial F/\partial X_0: \cdots: \partial F/\partial X_{n+1})$. 
Let $T_X:=\bigcap_{P \in X_{\rm sm}}T_PX$. 
If $T_X \ne \emptyset$, then $X$ is said to be strange and $T_X$ is called a strange center. 
A strange center is a linear space. 
It is well known that (the function field extension induced by) the projection from $T_X$ is separable if and only if $X$ is a cone with center $T_X$. 
Therefore, any strange variety is a cone with center $T_X$ if $p=0$. 

For distinct points $P, Q \in \mathbb P^{n+1}$, the line passing through $P$ and $Q$ is denoted by $\overline{PQ}$. 
For a point $P \in \mathbb P^{n+1}$, the projective space $\mathbb P^{n}$ parameterizes lines passing through $P$. 
Then, we can identify $\pi_P(Q)=\overline{PQ}$ for any point $Q \in X \setminus \{P\}$. 
If $P=(0:\cdots:0:1)$, then $\pi_P(X_0:\cdots:X_{n+1})=(X_0:\cdots:X_n)$, up to the projective equivalence of $\mathbb P^n$.      

We note the following Bertini theorem (see \cite[Theorem 1.1]{fulton-lazarsfeld}, or \cite[II. 6.1, Theorem 1]{shafarevich} and \cite[Lemma 5]{zariski}). 

\begin{lemma} \label{hyperplane section}
Let $P \not\in T_X$ (i.e., the projection $\pi_P: X \dashrightarrow \mathbb P^n$ is generically finite and separable). 
Then, for a general hyperplane $H \subset \mathbb P^{n+1}$ with $H \ni P$, the hyperplane section $X_H:=X \cap H$ is an integral scheme. 
\end{lemma}

Let $P$ be a Galois point for $X$. 
The Galois group induced by the Galois extension is denoted by $G_P(X)$, or simply $G_P$.   
Then, any element of $G_P(X)$ corresponds to a birational map from $X$ to itself. 
Let ${\rm Bir}(X)$ be the group consisting of all birational map from $X$ to itself. 
We can consider $G_P(X)$ as a subgroup of ${\rm Bir}(X)$. 
For $\sigma \in G_P(X)$, the maximal open subset of $X$ on which $\sigma$ is defined is denoted by $U_{\sigma}$.  
We set $U[P]=\bigcap_{\sigma \in G_P}U_{\sigma}$.  

\begin{lemma} \label{rigion}
Let $P \in \mathbb P^{n+1}$ be a Galois point and let $\sigma \in G_P$ be an induced birational map from $X$ to itself. 
Suppose that $H$ is a hyperplane such that $P \in H$ and $X_H$ is an integral scheme. 
Then, $X_H \cap U_{\sigma} \ne \emptyset$. 
Furthermore, if the multiplicity of $X_H$ at $P$ is less than $d$, then the restriction map $\sigma|_{X_H}$ is a birational map from $X_H$ to itself.  
\end{lemma} 

\begin{proof} 
For the first assertion, see the proof of \cite[V. Lemma 5.1]{hartshorne}. 
Now, $X_H$ corresponds to a regular point of codimension one in the scheme $X$.  
We can prove the first assertion by using a valuative criterion of properness for $X$. 
We prove the second assertion. 
Let $\tau$ be the inverse of $\sigma$. 
Then, $\sigma$ is an isomorphism from $U[P] \cap \sigma^{-1}U[P]$ to $U[P] \cap \tau^{-1}U[P]$. 
By definitions of $\pi_P$ and $\sigma$, $\sigma(X_H) \subset X_H$. 
If $U[P] \cap \sigma^{-1}U[P] \cap H \ne \emptyset$, we have the conclusion. 
Assume that $U[P] \cap \sigma^{-1}U[P] \cap H =\emptyset$. 
Then, $\sigma(U[P] \cap H) \subset (X \cap H) \setminus (U[P] \cap H)$. 
Since $U[P] \cap H$ is of dimension $n-1$ and $(X \cap H) \setminus (U[P] \cap H)$ is of dimension $\le n-2$, $\sigma^{-1}(\sigma(Q))$ is one-dimensional for a general point $Q$ of $X \cap H$. 
By definitions of $\pi_P$ and $\sigma$, $\overline{PQ} \subset X \cap H$.  
Then, $X_H$ is a cone with center $P$. 
Therefore, the multiplicity of $X_H$ at $P$ is equal to $d$. 
This is a contradiction. 
\end{proof}

\begin{lemma} \label{transitive}
Let $P$ be a Galois point for $X$ with multiplicity $m$ and let $Q, R \in X$ be points such that $\pi_P(Q)=\pi_P(R)$ and the intersection multiplicity of $X$ and $\overline{QR}$ at $P$ is equal to $m$. 
Assume that $Q \in U[P]$. 
Then, there exists $\sigma \in G_P$ such that $\sigma(Q)=R$.  
\end{lemma}

\begin{proof} 
Let $(X_0:\cdots:X_{n+1})$ be a system of homogeneous coordinates, let $x_i=X_i/X_0$ for $i=1, \ldots, n+1$, and let $F$ be the defining homogeneous polynomial of $X$. 
We can assume that $P=(0:\cdots:0:1)$ and $R=(1:0:\cdots:0)$ for a suitable system of coordinates. 
Then, the line passing through $P, Q, R$ is given by $X_1=\cdots=X_{n}=0$. 
Then, we can take $\pi_P(1:x_1:\cdots:x_{n+1})=(1:x_1:\cdots:x_n)$ and $\pi_P(Q)=\pi_P(R)=(1:0:\cdots:0)$. 
We have
$$F(X_0, \ldots, X_{n+1})=A_{d-m}X_{n+1}^{d-m}+\cdots+A_1X_{n+1}+A_0, $$ 
where $A_i \in K[X_0, \ldots, X_n]$. 
We define $f(x_1, \ldots, x_{n+1}):=F(1, x_1, \ldots, x_{n+1})$ and $a_i(x_1, \ldots, x_n):=A_i(1, x_1, \ldots, x_n)$.  
Since $f(R)=0$, $a_0(0, \ldots, 0)=0$. 
Since the intersection multiplicity of $X$ and $\overline{QR}$ at $P$ is $m$, $a_{d-m} (0, \ldots, 0)=A_{d-m}(1, \ldots, 0) \ne 0$. 

We consider a function $x_{n+1}$. 
Let $g=\prod_{\sigma} \sigma^{*}x_{n+1}$.
Then, we have $g(Q)=((-1)^{d-m}a_0/a_{d-m})(Q)=0$. 
Assume that $\sigma(Q) \ne R$ for any $\sigma \in G_P$. 
Then, $x_{n+1}(\sigma(Q)) \ne 0$ for any $\sigma \in G_P$. 
Therefore, $g(Q) \ne 0$. 
This is a contradiction. 
\end{proof}

\section{Proof of Theorem \ref{HyperplaneMain}}
\begin{proof}[Proof of Theorem \ref{HyperplaneMain}]
Suppose that $P$ is Galois for $X$. 
Then, $P \not\in T_X$. 
By Lemma \ref{hyperplane section}, for a general hyperplane $H \ni P$, $X_H$ is an integral scheme. 
Let $W_P \subset \check{\mathbb{P}}^{n+1}$ be the set of such hyperplanes.
Since $P \not\in T_X$, $W_P \setminus \gamma(X_{\rm sm}) \ne \emptyset$.  
We have $S_{X_H}=S_X \cap H$ for any $H \in W_P \setminus \gamma(X_{\rm sm})$, because, for a point $Q \in X_{\rm sm} \cap H$, $T_QX=H$ if and only if $X \cap H$ is singular at $Q$.   
Let $U_P \subset X_{\rm sm}$ be the set of all smooth points $Q$ such that the differential map of the projection $\pi_P$ at $Q$ is surjective, and let $\Sigma_P \subset \check{\mathbb P}^{n+1}$ be the set of all hyperplanes $H$ such that $H \subset (X \setminus U_P)$. 
Since $\pi_P$ is separable, $W_P \setminus \Sigma_P \ne \emptyset$. 
We have $P \not\in T_{X_H}$ for any $H \in W_P\setminus (\gamma(X_{\rm sm}) \cup \Sigma_P)$. 
Let $\Gamma_P$ be the (finite) set of all hyperplanes $H$ such that the multiplicity of $X \cap H$ at $P$ is less than $m$.  
We have $W_P \setminus (\gamma(X_{\rm sm}) \cup \Sigma_P \cup \Gamma_P) \ne \emptyset$, and any hyperplane $H$ in this set satisfies $(\star)$.   
We have assertion (i). 

Let $H$ be a hyperplane passing through $P$ and satisfying $(\star)$. 
We consider a homomorphism of groups 
$$ \phi: G_P \rightarrow G; \sigma \mapsto \sigma|_{X_H},  $$
where $G=\{ \sigma \in {\rm Bir}(X_H)|\sigma(X_H \cap l \setminus \{P\}) \subset X_H \cap l \mbox{ for a general line } l \mbox{ such that } P \in l \subset H  \}$. 
Since the multiplicity of $X_H$ at $P$ is $m <d$, it follows from Lemma \ref{rigion} and the definition of a Galois point that $\phi$ is well-defined. 
In addition,  by the condition $P \not\in T_{X_H}$, $X_H \cap l \setminus \{P\}$ consists of $d-m$ points for a general line $l \subset H$ containing $P$.   
It follows from Lemma \ref{transitive} that $\phi$ is injective. 
Since the order of $G$ is at most $d-m$, $\phi$ is an isomorphism. 
Then, $P$ is Galois for $X_H$. 
\end{proof} 

\begin{remark} 
A general hyperplane section for a Galois point which is {\it singular} was studied by T. Takahashi in his Ph.D. thesis \cite{takahashi2}. 
He proved that a Galois point $P \in S_X$ is also Galois for a general hyperplane section $X_H \ni P$, under the assumption that $p=0$ and $X \subset \mathbb{P}^3$ is a normal surface. 
\end{remark}

\section{Case of inner Galois points} 
As an application of Theorem \ref{HyperplaneMain}, we have the following. 

\begin{corollary} \label{reduction1} 
Assume that the condition (i) in Theorem \ref{InnerMain} holds. 
Then, for a general hyperplane $H$, there exists a non-empty Zariski open set $U_{X_H} \subset X_H$ such that $U_{X_H} \subset \Delta(X_H)$.
\end{corollary}

\begin{proof}
Let $U$ be an open set as in Theorem \ref{InnerMain}(i) and let $H$ be a general hyperplane. 
Since $H$ is general, we can assume that $U \cap H \ne \emptyset$. 
By Lemma \ref{hyperplane section}, $X_H$ is an integral scheme and $H$ satisfies the condition $(\star)$ for a general point $P \in X_H$.  
By Theorem \ref{HyperplaneMain}(ii), we can take $U_{X_H}=(U \cap X_H) \setminus (S_{X_H} \cup T_{X_H})$. 
\end{proof}

\begin{proof}[Proof of Theorem \ref{InnerMain}] 
We prove the implication (i) $\Rightarrow$ (ii). 
We use induction on dimension $n$. 
If $n=1$, then the assertion is nothing but a result of Fukasawa and Hasegawa \cite{fukasawa-hasegawa}. 
We consider the case where $n \ge 2$. 
Let $H \subset \mathbb P^{n+1}$ be a general hyperplane. 
It follows from Corollary \ref{reduction1} that there exists a Zariski open set $U_{X_H} \subset X_H$ such that $U_{X_H} \subset \Delta(X_H)$. 
By the assumption of induction, $p>0$, $d$ is a power of $p$, and $X_H \subset H \cong \mathbb P^{n}$ is projectively equivalent to $X_0^{d-1}X_1-X_2^d=0$.  
The Gauss map $\gamma_{X_H}$ for $X_H$ is given by $(-X_0^{d-2}X_1:X_0^{d-1}:0:\cdots:0)$, if $X_H \subset \mathbb P^n$ is given by $X_0^{d-1}X_1-X_2^d=0$. 
Therefore, we find that $T_{X_H}$ is a linear space of dimension $n-2$ with $T_{X_H} \not\subset X_H$, and a general fiber of $\gamma_{X_H}$ is a linear space of dimension $n-2$. 
Since $H$ is general and $T_QX_H=T_QX \cap H$ for any smooth point $Q \in X_H$, $T_X$ is a linear space of dimension $n-1$ with $T_X \not\subset X$, $\gamma_{X_H}$ coincides with the restriction map $\gamma|_{X_H}$, and a general fiber of $\gamma$ is a linear space of dimension $n-1$.

Let $P \in X$ be a general point. 
Since the linear spaces $\gamma^{-1}(\gamma (P))$ of dimension $n-1$ and $T_X$ of dimension $n-1$ are contained in the projective space $T_PX$ of dimension $n$, these intersect along a linear space $L_P$ of dimension $n-2$. 
If $L_{P'} \ne L_{P}$ for a general point $P' \in X$, then $T_X$ is contained in $X$, since $L_{P'} \subset T_X$. 
This is a contradiction. 
Therefore, $L_{P'}=L_P$. 
Then, $X$ is a cone with a $(n-2)$-dimensional center $L$. 
For a suitable system of coordinates, we can assume that $L$ is defined by $X_0=X_1=X_2=0$ and $X$ is defined by $F(X_0, X_1, X_2)=0$. 
We can assume that $H$ is defined by $X_{n+1}-(a_0X_0+\cdots+a_{n}X_n)=0$. 
Then, $X_H$ is given by the same equation $F=0$ and there exists a linear transformation $\phi:H \rightarrow H$ such that $\phi(X_H)$ is defined by $F_1:=X_0^{d-1}X_1-X_2^d=0$. 
Then, $\phi(L \cap H)=L \cap H$. 
Therefore, $\phi$ gives an automorphism of the sublinear system $\langle X_0, X_1, X_2 \rangle$ of $H^0(\mathbb P^n, \mathcal{O}(1))$ which is given by $L \cap H$. 
This implies that $X$ is projectively equivalent to the hypersurface defined by $X_0^{d-1}X_1-X_2^d=0$. 

We consider the implication (ii) $\Rightarrow$ (i). 
Let $F=X_0^{p^e-1}X_1-X_2^{p^e}$ be the defining polynomial. 
It is not difficult to check that the singular locus $S_X$ of $X$ is given by $X_0=X_2=0$. 
Let $P \in X \setminus S_X$. 
Then, $P=(1:b_1: \cdots: b_{n+1})$ for some $b_1, \ldots, b_{n+1} \in K$ with $b_1=b_2^{p^e}$. 
The projection $\pi_P$ is given by $(X_1-b_1X_0:\cdots:X_{n+1}-b_{n+1}X_0)$. 
Let $\hat{X_i}=X_i-b_iX_0$. 
Then, $F(X_0, \hat{X_1}+b_1X_0, \ldots, \hat{X}_{n+1}+b_{n+1}X_0)=(\hat{X}_1+b_1X_0)X_0^{p^e-1}-(\hat{X_2}+b_2X_0)^{p^e}=\hat{X}_1X_0^{p^e-1}-\hat{X}_2^{p^e}=F(X_0, \hat{X}_1, \ldots, \hat{X}_n)$. 
Then, $\pi_P=(X_0:\hat{X}_1:\cdots:\hat{X}_{n+1})$. 
Therefore, we have a field extension $K(x_0, x_2, \ldots, x_{n+1})/K(x_2, \ldots, x_{n+1})$ with a relation $F(x_0, 1, x_2, \ldots, x_{n+1})=x_0^{p^e-1}-x_2^{p^e}=0$. 
It is not difficult to check that this is a Galois extension, which is cyclic of degree $p^e-1$. 
Therefore, we have $\Delta(X)=X \setminus \{X_0=X_2=0\}$. 
\end{proof}

\section{Case of outer Galois points} 
As an application of Theorem \ref{HyperplaneMain}, we have the following. 

\begin{corollary} \label{reduction2}
Assume the condition (i) in Theorem \ref{OuterMain} holds. 
Then, for a general hyperplane $H$, there exists a non-empty Zariski open set $U_{Y_H} \subset Y_H$ such that $U_{Y_H} \subset \Delta'(X_H)$. 
\end{corollary}
 
\begin{proof} 
Let $Y, U_Y$ be as in Theorem \ref{OuterMain}(i) and let $H$ be a general hyperplane. 
Since $H$ is general, we can assume that $U \cap H \ne \emptyset$.  
By Lemma \ref{hyperplane section}, $X_H$ and $Y_H$ are integral and $H$ satisfies the condition $(\star)$ for a general point $P \in Y_H$. 
By Theorem \ref{HyperplaneMain}(ii),  we can take $U_{Y_H}=(U_Y \cap H) \setminus T_{X_H}$.
\end{proof} 

\begin{lemma} \label{n=1}
Let $n=1$. 
Assume that there exists an irreducible plane curve $Y \subset \mathbb P^2$ and a non-empty open set $U_Y \subset Y$ such that $U_Y \subset \Delta'(X)$. 
Then, 
\begin{itemize}
\item[(1)] $Y$ is a line, and   
\item[(2)] if we take a linear transformation $\phi$ such that $\phi(Y)$ is defined by $X_0=0$, then the defining polynomial $\phi(X)$ is of the form $\sum_{j=0}^{e}\sum_{i=1}^{2}\alpha_{ij}x_i^{p^{j}}+c=0$, where $\alpha_{ij}, c \in K$. 
\end{itemize}
\end{lemma} 

\begin{proof}
Let $(X_0:X_1:X_{2})$ be a system of homogeneous coordinates and let $x_i=X_{i}/X_0$ for $i=1, 2$. 
By a result of Fukasawa \cite{fukasawa2}, $Y$ is a line and there exists a linear transformation $\psi$ such that $\psi(X)$ is defined by $f:=\sum_{j=0}^{e}\sum_{i=1}^{2}\alpha_{ij}x_i^{p^{j}}=0$. 
Then $\psi(Y)$ is defined by $X_0=0$. 
Let $\phi$ be a linear transformation as in assumption (2). 
Then, $\phi(X)$ is given by $(\psi \circ \phi^{-1})^*f=0$. 
Let $\tilde{x}_i:=(\psi\circ\phi^{-1})^*x_i=\beta_{i0}+\beta_{i1}x_1+\beta_{i2}x_2$, where $\beta_{ij} \in K$, for $i=1, 2$.  
Then, $\phi(X)$ is given by $f(\tilde{x}_1, \tilde{x}_2)=0$, where $f(\tilde{x}_1, \tilde{x}_2)=\sum_{j=0}^{e}\sum_{i=1}^{2}\gamma_{ij}x_i^{p^{j}}+c$ for some $\alpha_{ij}, c \in K$. 
\end{proof}

\begin{proof}[Proof of Theorem \ref{OuterMain}] 
We consider the following condition ($P_n$): If $X \subset \mathbb P^{n+1}$ is an irreducible hypersurface, and there exists an irreducible hypersurface $Y \subset \mathbb P^{n+1}$ and a non-empty open set $U_Y$ of $Y$ such that $U_Y \subset \Delta'(X)$, then 
\begin{itemize}
\item[(1)] $Y$ is a hyperplane, and   
\item[(2)] if we take a linear transformation $\phi$ such that $\phi(Y)$ is defined by $X_0=0$, then the defining polynomial of $\phi(X)$ is of the form $\sum_{j=0}^{e}\sum_{i=1}^{n+1}\alpha_{ij}x_i^{p^{j}}+c=0$, where $\alpha_{ij}, c \in K$. 
\end{itemize}
We prove ($P_n$).  
We use induction on dimension $n$. 
If $n=1$, then the assertion is nothing but Lemma \ref{n=1}. 
We consider the case where $n \ge 2$. 
Let $H \subset \mathbb P^{n+1}$ be a general hyperplane. 
It follows from Corollary \ref{reduction2} that there exists a non-empty Zariski open set $U_{Y_H} \subset \Delta'(X_H)$. 
By the assumption of induction, $p>0$, $d$ is a power of $p$, and $X_H \subset H \cong \mathbb P^{n-1}$ is projectively equivalent to $\sum_{j=0}^{e}\sum_{i=1}^n\alpha_{ij}x_i^{p^j}+c=0$.  
Since $H$ is general, $Y$ is a hyperplane. 
We have result (1) of $(P_n)$. 

We take a linear transformation $\phi$ such that $\phi(Y)$ is defined by $X_{0}=0$. 
Let $P=(1:a_1:\ldots:a_{n+1}) \in X$ be a general point and let $H$ be a general hyperplane passing through $P$. 
Then, there exists an open set $U_{Y_H} \subset Y_H$ such that $U_{Y_H} \subset \Delta'(X_H)$. 
If we take a linear transformation $\psi=(X_0:X_1-a_1X_0:\cdots:X_{n+1}-a_{n+1}X_{0})$, then $\psi(P)=(1:0:\cdots:0)$, $\psi(\phi(Y))$ is defined by $X_0=0$ and $\psi(H)$ is defined by $X_{n+1}-b_1X_1+\cdots+b_nX_n=0$ for some $b_i \in K$.   
Let $F(X_0, \ldots, X_{n+1})$ be the defining polynomial of $\psi(X)$, let $f=F(1, x_1, \ldots, x_{n+1})$ and let $\tilde{x}=b_1x_1+\cdots+b_nx_n$. 
Then, $f(0, \ldots, 0)=0$ since $(1:0:\cdots:0) \in \psi(X)$.  
Since $\psi(X) \cap \psi(H)$ satisfies the condition $(P_{n-1})$ by induction, $g(x_1, \ldots, x_n):=f(1, x_1, \ldots, x_{n}, \tilde{x})=\sum_{i,j}\beta_{ij}x_i^{p^j}$ for some $\beta_{ij} \in K$, similar to the proof of Lemma \ref{n=1}.
If $f$ has a term of degree not equal to some power of $p$, $g$ has such a term for a general hyperplane $H$ with $P \in H$. 
Therefore, $f$ has only terms of degree equal to some power of $p$. 
Let $f_{p^i}$ be the component of $f$ of degree $p^i$ for $i=0, \ldots, e$. 
Then, $f_{p^i}(x_1, \ldots, x_n, \tilde{x})$ be the component of $g$ of degree $p^i$. 
By condition (2) of ($P_{n-1}$), $f_{p^i}(x_1, \ldots, x_n, \tilde{x})$ must be of the form $h^{p^i}$ for some linear polynomial $h(x_1, \ldots, x_n)$. 
Since $H$ is general, $f_{p^i}(x_1, \ldots, x_{n+1})$ must be of the form $h^{p^i}$ for some $h(x_1, \ldots, x_{n+1})$. 
Therefore, $\psi(X)$ is given by the polynomial as in condition (2) of ($P_n$) and hence, $\phi(X)$ also.

The implication (ii) $\Rightarrow$ (i) is derived from Proposition \ref{distribution} below. 
\end{proof}

\begin{proposition}[cf. \cite{fukasawa2}, Propositions 2 and 3] \label{distribution}
Let $X$ be an irreducible hypersurface defined by the equation in Theorem \ref{OuterMain}(ii) and let $H_0$ be the hyperplane defined by $X_0=0$. 
Then, we have the following. 
\begin{itemize}
\item[(i)] $S_X$ and $T_X$ are linear spaces of dimension $n-1$ which are contained in $H_0$. 
\item[(ii)] $\Delta'(X)=H_0\setminus (S_X \cup T_X)$, and all points in $S_X \setminus T_X$ are Galois. 
(Here, we consider a point $P$ with $\pi_P^*(\mathbb{P}^n)=K(X)$ as a Galois point.)
\item[(iii)] For any Galois point $P \in H_0 \setminus T_X$, any birational map induced by $G_P$ is a restriction of a linear transformation of $\mathbb P^2$. 
\item[(iv)] For any Galois point $P \in H_0 \setminus T_X$, the Galois group $G_P$ is isomorphic to $(\mathbb Z/p\mathbb Z)^{\oplus m}$ for some $m \le e$. 
\item[(v)] $\Delta(X)=X_{\rm sm}$ if $X$ is projectively equivalent to the hypersurface defined by $X_0^{p^e-1}X_1-X_2^{p^e}=0$, and $\Delta(X)=\emptyset$ otherwise. 
\end{itemize}
\end{proposition} 

\begin{proof} 
Let $F=\sum_{j=0}^{e}\sum_{i=1}^{n+1}\alpha_{ij}X_0^{p^e-p^j}X_i^{p^j}$ be the defining polynomial. 
It is not difficult to check that $S_X$ is given by $X_0=\sum_{i=1}^{n+1}\alpha_{ie}X_i^{p^e}=0$ and $T_X$ is given by $X_0=\sum_{i=1}^{n+1}\alpha_{i0}X_i=0$.   
Therefore, we have $(i)$. 

We prove that all points in $H_0 \setminus T_X$ are Galois. 
There exists $i$ such that $\alpha_{i0} \ne 0$. 
We can assume that $i=1$. 
Let $\hat{X}_1=\sum_{i=1}^{n+1}\alpha_{i0}X_i$. 
Then, $\phi(X_0, \ldots, X_{n+1})=(X_0, \hat{X_1}, X_2 \ldots, X_{n+1})$ is a linear transformation and $\phi(T_X)$ is given by $X_0=X_1=0$. 
By considering $\phi(X)$ as $X$, we can assume that $T_X$ is given by $X_0=X_1=0$. 
Let $P \in H_0 \setminus T_X$. 
Then, $P=(0:1:b_2:\cdots: b_{n+1})$ for some $b_2, \ldots, b_{n+1}$. 
The projection $\pi_P$ is given by $(1:x_2-b_2x_1:\cdots:x_{n+1}-b_{n+1}x_1)$. 
Let $\hat{x}_i=x_i-b_ix_1$. 
Then, we have a field extension $K(x_1, \hat{x}_2, \ldots, \hat{x}_{n+1})/K(\hat{x}_2, \ldots, \hat{x}_{n+1})$ with a relation $g(x_1, \hat{x}_2, \ldots, \hat{x}_{n+1})=F(1, x_1, \hat{x}_2+b_2x_1, \ldots, \hat{x}_{n+1}+b_{n+1}x_1)=0$.  
By the form of $F$, $g$ is of the form $\sum_{i,j}\beta_{ij} x_i^{p^j}$ for some $\beta_{ij} \in K$. 
Since $P \not\in T_X$, this extension is Galois of degree $p^m$ for some $m \le e$ and the Galois group is isomorphic to $(\mathbb Z/p\mathbb Z)^{\oplus m}$ (see \cite[pp. 117--118]{stichtenoth}). 
Therefore, we have $H_0 \setminus (S_X \cup T_X) \subset \Delta'(X)$, and any point in $S_X \setminus T_X$ is Galois. 
By considering the form $g$, we find that assertions (iii) and (iv) hold for all Galois points in $H_0\setminus T_X$. 

We prove that $\Delta'(X) \subset H_0 \setminus (S_X \cup T_X)$. 
If this is proved, then we have (ii), (iii) and (iv). 
If $n=1$, then this is a result of Fukasawa (see \cite[Proposition 2]{fukasawa2}). 
Assume that $n \ge 2$ and there exists an outer Galois point $P \not\in H_0$ for $X$. 
By Theorem \ref{HyperplaneMain}, there exists a hyperplane $H \ni P$ such that $\{P\} \cup (\Delta'(X) \cap H_ 0 \cap H) \setminus T_{X_H} \subset \Delta'(X_H)$. 
When $n=2$, $X_H$ is a curve and this is a contradiction.  
By using induction, we have a contradiction for any $n \ge 2$.  

We prove (v). 
Assume that $P \in \Delta(X)$. 
Since $X \cap H_0=S_X$, $P \in X \setminus H_0$. 
Let $P' \in X \setminus H_0$ such that the line $\overline{PP'}$ intersects the set $\Delta'(X)$. 
Let $R \in \overline{PP'} \cap \Delta'(X)$. 
Since any element of $G_R$ is a linear transformation, it follows from Lemma \ref{transitive} that there exists $\sigma \in G_R$ such that $\sigma(P)=P'$. 
Since $P$ is Galois, $P'$ is also Galois. 
Therefore, if one inner Galois point exists, then almost all points of $X$ are inner Galois points. 
It follows from Theorem \ref{InnerMain} that $X$ is projectively equivalent to the hypersurface defined by $X_0^{p^e-1}X_1-X_2^{p^e}=0$ and $\Delta(X)=X\setminus \{X_0=X_1=0\}=X_{\rm sm}$.  
\end{proof}

\end{document}